\documentclass[11pt, reqno]{amsart}
\usepackage{cases}
\usepackage[square, comma, sort&compress, numbers]{natbib}
\usepackage{float}
\usepackage{pifont}
\usepackage{bbding}
\usepackage{amssymb}
\usepackage{mathrsfs}
\usepackage{subfigure}
\usepackage{caption}
\usepackage{geometry}
\usepackage{amsmath, amsthm, amscd, amsfonts, amssymb, graphicx, color}
\usepackage{lineno,hyperref}
\newtheorem{thm}{Theorem}
\newtheorem{theorem}{Theorem}[section]
\newtheorem{lemma}[theorem]{Lemma}

\theoremstyle{definition}

\newtheorem{problem}{Problem}

\newtheorem{note}{Note}
\theoremstyle{remark}
\numberwithin{equation}{section}
\DeclareMathOperator{\RE}{Re}

\newcommand{\ID}{\mathbb{D}}
\newcommand{\IR}{{\mathbb R}}

\newcommand{\ds}{\displaystyle}
\begin{document}
\thispagestyle{empty} \setcounter{page}{1}


\title[Univalent harmonic right half-plane mappings]
{Univalency of convolutions of univalent harmonic right half-plane mappings}
\author[Z. Liu]{Zhihong Liu }
\address{Z.Liu, School of Mathematics and Econometrics, Hunan University, Changsha 410082, Hunan, People's Republic of China.
\vskip.03in College of Mathematics, Honghe University, Mengzi 661199, Yunnan, People's Republic of China.}
\email{\textcolor[rgb]{0.00,0.00,0.84}{liuzhihongmath@163.com}}

\author[S. Ponnusamy]{Saminathan Ponnusamy}
\address{S. Ponnusamy, Indian Statistical Institute (ISI), Chennai Centre,
SETS (Society for Electronic Transactions and Security), MGR Knowledge City, CIT Campus,
Taramani, Chennai 600 113, India.}
\email{\textcolor[rgb]{0.00,0.00,0.84}{samy@isichennai.res.in, samy@iitm.ac.in}}




\begin{abstract}
We consider the convolution of half-plane harmonic mappings with respective dilatations $(z+a)/(1+az)$ and $e^{i\theta}z^{n}$, where $-1<a<1$ and $\theta\in\mathbb{R},n\in\mathbb{N}$. We prove that such convolutions are locally univalent for $n=1$, which solves an open problem of
Dorff et. al (see \cite[Problem~3.26]{Bshouty2010}). Moreover, we provide some numerical computations to illustrate that such convolutions are not
univalent for $n\geq 2$.
\end{abstract}
\maketitle

\section{Introduction and Preliminaries}
The subject of this study is the convolution of functions in the class  $\mathcal{H}$ of all complex-valued harmonic mappings $f=h+\overline{g}$
defined on the open unit disk $\mathbb{D}=\{z\in\mathbb{C}:\,|z|<1\}$ normalized by the condition $h(0)=g(0)=h'(0)-1=0$,
where $h$ and $g$ are analytic in $\mathbb{D}$. For two such harmonic mappings $f$ and $F=H+\overline{G}$ in $\mathcal{H}$ with power
series of the form
$$f(z)
=z+\sum_{n=2}^{\infty}a_{n}z^{n}+\sum_{n=1}^{\infty}\overline{b_{n}}\overline{z}^{n} ~\mbox{ and }~
F(z)
=z+\sum_{n=2}^{\infty}A_{n}z^{n}+\sum_{n=1}^{\infty}\overline{B_{n}}\overline{z}^{n},
$$
we define the harmonic convolution (or Hadamard product) as follows:
$$f*F=h*H+\overline{g*G}=z+\sum_{n=2}^{\infty}a_{n}A_{n}z^{n}+\sum_{n=1}^{\infty}\overline{b_{n}B_{n}}\overline{z}^{n}.$$
By Lewy's theorem~\cite{Lewy1936}, $f\in \mathcal{H}$ is locally univalent and sense-preserving if and only if $J_{f}(z)>0$ in $\mathbb{D}$,
where $J_{f}(z)= |h'(z)|^{2}-|g'(z)|^{2}$ denotes the Jacobian of $f$. The condition $J_{f}(z)> 0$ is equivalent to the existence of an analytic function
$\omega$, called the \textit{dilatation} of $f$, given by $\omega(z)=g'(z)/h'(z)$ with $|\omega(z)| < 1$ for all $z\in\mathbb{D}$, where
$h'(z)\neq 0$ in $\ID$. Denote by $\mathcal{S}_{H}$ the class of all sense-preserving harmonic univalent mappings $f=h+\overline{g}\in\mathcal{H}$ and by $\mathcal{S}_{H}^{0}$ the subclass of mappings $f\in\mathcal{S}_{H}$ such that $f_{\overline{z}}(0)=0$. For many basic results about several geometric subclasses of  $\mathcal{S}_{H}$ and $\mathcal{S}_{H}^{0}$, we refer to the
article of Clunie and Sheil-Small~\cite{Clunie} (or see~\cite{Dorff2012B,Duren, SaRa2013}).
Denote by $\mathcal{K}_{H}^{0}$, the class of functions in $\mathcal{S}_{H}^{0}$ which have convex images, and functions in
$\mathcal{K}_{H}^{0}$ are called convex.

Unlike the case of analytic mappings, the properties of harmonic convolutions are not so regular in the sense that
the convolution of two convex harmonic mappings is not necessarily even locally univalent in $\ID$. However, these convolutions
do exhibit some fascinating properties. In recent years,  properties of convolutions
of harmonic mappings were investigated by a number of authors, see for
example~\cite{Dorff2001,Dorff2012,Liu1,Li1,Li2,
Wang2016} and the references therein.
In~\cite{Muhanna1987,Dorff1999} and~\cite{Hengartner1987}, explicit descriptions of half-plane mappings and
strip mappings are given.

Recall that a domain $\Omega\subset \mathbb{C}$ is said to be convex in the horizontal direction~(CHD) if the intersection of $\Omega$ with each horizontal line is connected (or empty). We now recall one of the fundamental results, called \textit{shearing theorem}, due to Clunie and Sheil-Small~\cite{Clunie}.

\begin{thm}\label{thmA}
Let $f=h+\overline{g}$ be harmonic and locally univalent in $\mathbb{D}$. Then $f$ is univalent and its range is {\rm CHD} if and only if $h-g$ is univalent and its range is {\rm CHD}.
\end{thm}

A function  $f=h+\overline{g} \in {\mathcal S}_H$ is said to be a slanted half-plane mapping with $\gamma$ ($0\leq\gamma<2\pi$) if $f$
maps $\ID$ onto  $H_\gamma :=\{w:\,{\rm Re\,}(e^{i\gamma}w) >-(1+a)/2\}$, where $-1<a<1$.
Using the shearing method due to Clunie and Sheil-Small \cite{Clunie}  and the Riemann mapping theorem,
it is easy to see that such a mapping has the form
\begin{equation}\label{li2-eq3}
h(z)+e^{-2i\gamma}g(z)=\frac{(1+a)z}{1-e^{i\gamma}z}.
\end{equation}
Note that $h(0)=g(0)=h'(0)-1=0$ and $g'(0)=a$.
The class of all slanted half-plane mappings with $\gamma$ is denoted by ${\mathcal S}(H_{\gamma})$. Clearly,
each $f\in {\mathcal S}(H_{\gamma})$ obviously belongs to the convex family $\mathcal{K}_{H}$ but not necessarily in $\mathcal{K}_{H}^{0}$
unless $a=0$. It is evident that there are infinitely many \textit{slanted half-plane mappings} with a fixed $\gamma$.
We denote by $\mathcal{S}^{0}(H_{\gamma})$ if $a\in(-1,1)$ in $\mathcal{S}(H_{\gamma})$ is taken to be zero.
At this point, it is worth recalling that functions $f=h+\overline{g} \in {\mathcal S}^0(H_{0})$ (i.e. $f\in{\mathcal S}(H_{\gamma})$ with $a=0$) are usually referred to as the slanted half-plane mappings with $\gamma$ and such mappings by \eqref{li2-eq3} obviously assume the form
\begin{equation}\label{li2-eq3a}
h(z)+e^{-2i\gamma}g(z)=\frac{z}{1-e^{i\gamma}z}
\end{equation}
so that $\gamma =0$ reduces to the corresponding right half-plane mappings.
For example, if $f_{0}=h_{0}+\overline{g_{0}} \in {\mathcal S^0}(H_{0})$ with the dilatation $\omega_{0}=g'_{0}/h'_{0}=-z$
then the shearing theorem of  Clunie and Sheil-Small  quickly yields (see \cite{Clunie})
\begin{equation}\label{eqHRN}
h_{0}=\frac{z-\frac{1}{2}z^2}{(1-z)^2}\quad{\rm and}\quad g_{0}=\frac{-\frac{1}{2}z^2}{(1-z)^2}.
\end{equation}
The function $f_0$ is extremal for the coefficient inequality for functions in  ${\mathcal K}_H^0$.
\begin{thm}{\rm (\cite[Theorem 2]{Dorff2012})}\label{ThmA}
If $f_k\in {\mathcal S^0}(H_{\gamma_k})$, $k=1,\ 2$, and $f_1\ast
f_2$ is locally univalent in $\ID$, then $f_1\ast f_2$ is convex in
the direction $-(\gamma_1 +\gamma_2)$.
\end{thm}

Theorem \ref{ThmA} generalizes the result of Dorff \cite[Theorem 5]{Dorff2001} who proved originally the same with $\gamma _1=\gamma _2=0$.

Generally, it is not easy to verify the local univalency of $f_1\ast f_2$ in $\ID$.
In~\cite{Li2}, Li and Ponnusamy obtained the following result as a generalization of \cite[Theorem 3]{Dorff2001}.


\begin{thm}\label{thmC}
Let $f=h+\overline{g}\in  {\mathcal S^0}(H_\gamma)$ with the
dilatation $\omega(z)=e^{i\theta}z^n$, where  $n=1,2$ and $\theta\in\IR$.
Then $f_0\ast f\in {\mathcal S}_H^0$ and is convex in the direction $-\gamma$.
\end{thm}

In 2010, Bshouty and Lyzzaik~\cite{Bshouty2010} brought out a collection of open problems and conjectures
on planar harmonic mappings, proposed by many colleagues throughout the past quarter of a century.  Dorff et. al
(see \cite[Problem~3.26]{Bshouty2010}) posed the following open question.

\begin{problem}{(M.~Dorff, M.~Nowak and M.~Woloszkiewicz)}\label{prob}
\begin{enumerate}
\item[\textbf{(a)}] Let $f_{0}=h_{0}+\overline{g_{0}} \in {\mathcal S}^{0}(H_{0})$ be as above with the dilatation $\omega_{0}=g'_{0}/h'_{0}=-z$,
and  $f=h+\overline{g}\in {\mathcal S}(H_{0})$ with the dilatation $\omega(z)=(z+a)/(1+az)$,  $a\in(-1,1)$.
Then $f_{0}*f\in S_{H}^{0}$ and is convex in the direction of the real axis. Determine the other values of $a\in \mathbb{D}$
for which the corresponding result holds.

\item[\textbf{(b)}] Let $f_{n}\in {\mathcal S^0}(H_{0})$ with the  dilatations $\omega_{n}(z)=e^{i\theta}z^{n}~(\theta\in\mathbb{R},n\in\mathbb{N})$. Determine the values of $n$ for which $f_{n}*f$ are univalent.
\end{enumerate}
\end{problem}

One of the proposers of the above problems communicated to the first author about a typo in~\cite[{\textbf{Problem 3.26(b)}}]{Bshouty2010}. Again a typo in~\cite[{\textbf{Problem 3.26(a)}}]{Bshouty2010} is now corrected and the corrected formulation is stated in Problem~\ref{prob}\textbf{(b)}.
Our primary aim in this paper is to present a solution to this open problem.
However, Problem~\ref{prob}\textbf{(a)} has been solved by Li and  Ponnusamy~\cite{Li1}, and in \cite{Jiang2015,Li2} the same has been solved in a more general setting which led to further investigation and interest in this topic.
\begin{note}\label{n1}
In~\cite[Theorem 4]{Dorff2012}, the following result was shown under the assumption that $f=h+\overline{g}\in\mathcal{K}_{H}^{0}$ with $h(z)+g(z)=z/(1-z)$ and $\omega(z)=\frac{z+a}{1+az}$ with $a\in(-1,1)$. Unfortunately, the first condition gives $g'(0)=0$ while the second condition gives $g'(0)=a$. In view of this reasoning, we reformulate their result in the form which is needed in our proof.
\end{note}

\begin{theorem} \label{univalent4}
Let $f\in {\mathcal S}(H_0)$, i.e.,  $f=h+\overline{g}\in  {\mathcal K}_{H}$ with $h(z)+g(z)=(1+a)\frac{z}{1-z}$, and $\omega(z)=\frac{z+a}{1+a z}$ with $a\in(-1,1)$.
Then $f_0\ast f\in {\mathcal S}_H^0$ and is convex in the direction of the real axis.
\end{theorem}
\begin{proof}
In view of~\cite[Theorem A]{Dorff2001} and Lewy's Theorem, we just need to show that the dilatation $\widetilde{\omega}$ of $f_0\ast f$ satisfies the
condition $|\widetilde{\omega}(z)|<1$ for all $z\in \mathbb{D}$.
By assumption, we easily see that
$$h'(z)=\frac{1+a}{(1+\omega(z))(1-z)^2}\quad{\rm and}\quad h''(z)=\frac{(1+a)\left[2(1+\omega(z))-\omega'(z)(1-z)\right]}{(1+\omega(z))^2(1-z)^3}
$$
and the rest of the proof follows as in \cite[Theorem 4]{Dorff2012}. For the sake of completeness we include the necessary details here. Indeed,
since $g'(z)=\omega(z)h'(z)$, we have
\begin{equation*}
\begin{split}
\widetilde{\omega}(z)&=\frac{(g_{0}\ast g)'(z)}{(h_{0}\ast h)'(z)}\\
&=-z\frac{\omega'(z)h'(z)+\omega(z)h''(z)}{2h'(z)+zh''(z)}\\
&=-z\frac{p(z)}{p^{*}(z)}=-z\frac{(z+A)(z+B)}{(1+\overline{A}z)(1+\overline{B}z)},
\end{split}
\end{equation*}
where
$$p(z)=z^2+\frac{1+3a}{2}z+\frac{1+a}{2},  \quad  p^{*}(z)=z^2 \overline{p(1/\overline{z})}=1+\frac{1+3a}{2}z+\frac{1+a}{2}z^2,
$$
and $A, B$ are two zeros of $p(z)$. By using Cohn's rule,
$$q_{1}(z)=\frac{\overline{a_{2}}p(z)-a_{0}p^{*}(z)}{z}=\frac{(a+3)(1-a)}{4}\left(z+\frac{1+3a}{a+3}\right).$$
So $q_{1}(z)$ has one zero at $z_{0}=-\frac{1+3a}{a+3}$ which is on the unit disk for $-1<a<1$. Thus, $|A|,|B|<1$.
\end{proof}

\section{Main Results}

We now state and prove our main result which solves Problem~\ref{prob}\textbf{(b)} for $n=1$.

\begin{theorem}\label{thm-ff1}
Let  $f=h+\overline{g}\in{\mathcal S}(H_{0})$ with $h+g=(1+a)z/(1-z)$ and the dilatation $\omega(z)=(z+a)/(1+a z)$, where $-1<a<1$, and
$f_1 =h_{1}+\overline{g_{1}}\in{\mathcal S^0}(H_{0})$ with dilatation $\omega_{1}(z)=e^{i\theta}z~(\theta\in\mathbb{R})$.
Then $f_{1}*f$ is locally univalent and convex in the horizontal direction.
\end{theorem}

The problem of determining other values of $a\in \mathbb{D}$ and $|\epsilon|=1$ with $\omega(z)=\epsilon (z+a)/(1+az)$ in Theorem \ref{thm-ff1}
remains open. As in the case of ~\cite[Theorem 1.3]{Li2}, establishing the analog of Theorem \ref{thm-ff1} for slanted half-plane mappings
is another problem which needs further investigation.  Note also that Theorem \ref{thm-ff1} for $\theta =\pi$ is contained in Theorem \ref{univalent4}.

To prove our main result, we need the following lemmas.



\begin{lemma}{\rm(\cite[5.8 Corollary]{Clunie})}\label{lemC}
If $f=h+\overline{g}$ is convex, then $\left|\frac{g(z_{1})-g(z_{2})}{h(z_{1})-h(z_{2})}\right|<1$ for all $z_1,z_2\in \mathbb{D}$.
\end{lemma}

\begin{lemma}{\rm(\cite[Lemma 3]{Romney2013})}\label{lemD}
Let $f:\mathbb{D}\to \mathbb{C}$ be nonconstant and analytic, where $f(\overline{\mathbb{D}})$ omits some point $w\in\{z:\RE\,z <0\}$. Suppose that $\widehat{f}(e^{it})=\lim_{z\to e^{it}}f(z)$ exists for all $t\in \mathbb{R}$ (where possibly $\widehat{f}(e^{it})=\infty$). If $\RE\,\{\widehat{f}(e^{it})\}\geq 0$ for all $t$ such that $\widehat{f}(e^{it})$ is finite, then $\RE\,\{f(z)\}>0$ for all $z\in \mathbb{D}$.
\end{lemma}

We remark that Lemma \ref{lemD} is another convenient formulation of maximum modulus theorem for analytic functions.


\begin{proof}[Proof of Theorem~\ref{thm-ff1}] Conclusion of Theorem~\ref{thm-ff1} holds when  $\theta=\pi$ (see Theorem \ref{univalent4}) and thus,
we will assume throughout the discussion that $\theta\neq\pi$. Now, by \eqref{li2-eq3a} and the assumption on $f$, we have
$$h'(z)+g'(z)=\frac{1+a}{(1-z)^2} ~\mbox{ and }~\omega(z)=\frac{ g'(z)}{h'(z)}=\frac{z+a}{1+az}.
$$
Solving these two equations for $h'$ and $g'$ and then integrating the resulting equations give
\begin{equation}\label{eqhg}
h(z)=\frac{1+a}{2}\,\frac{z}{1-z}+\frac{1-a}{4}\log\left(\frac{1+z}{1-z}\right) ~\mbox{ and }~ g(z)=\frac{1+a}{2}\,\frac{z}{1-z}-\frac{1-a}{4}\log\left(\frac{1+z}{1-z}\right).
\end{equation}
Again, by the assumption on $f_1$, we have
$$h_{1}(z)+g_{1}(z)=\frac{z}{1-z}~\mbox{ and }~\omega_{1}(z)=\frac{ g'_{1}(z)}{h'_{1}(z)}=e^{i\theta}z.
$$
Solving these two equations gives
\begin{equation}\label{eqf1}
\begin{split}
h_{1}(z)&=\frac{1}{1+e^{i\theta}}\,\frac{z}{1-z}+\frac{e^{i\theta}}{(1+e^{i\theta})^2}\log\left(\frac{1+e^{i\theta}z}{1-z}\right), ~\mbox{and}\\
g_{1}(z)&=\frac{e^{i\theta}}{1+e^{i\theta}}\,\frac{z}{1-z}-\frac{e^{i\theta}}{(1+e^{i\theta})^2}\log\left(\frac{1+e^{i\theta}z}{1-z}\right).
\end{split}
\end{equation}
For any $F(z)=\sum_{n=1}^{\infty}c_{n}z^{n}$ analytic in $\ID$, we see that $(z/(1-z))*F(z)=F(z)$ and
\begin{equation}\label{eqhgC}
\log\left(\frac{1-xz}{1-z}\right)*F(z) =\int_0^z\frac{F(t)-F(xt)}{t}\,dt\quad \mbox{for $|x|\leq 1$, $x\neq 1$}.
\end{equation}
In view of equations \eqref{eqhg} and \eqref{eqf1}, a computation gives
 \begin{equation*}
\begin{split}
h(z)*h_{1}(z)&=\frac{1+a}{2}h_{1}(z)+\frac{1-a}{4}\int _0^z \frac{h_{1}(t)-h_{1}(-t)}{t}\,dt\\
&=\frac{1+a}{2} \left [\frac{1}{\left(1+e^{i \theta }\right)}\,\frac{z}{1-z}+\frac{e^{i \theta } }{\left(1+e^{i \theta }\right)^2}\log \left(\frac{1+e^{i \theta } z}{1-z}\right)\right]+\frac{(1-a)e^{i \theta }}{4 \left(1+e^{i \theta }\right)^2}\bigg [ \text{Li}_2(z) \\
&\qquad- \text{Li}_2(-z) + \text{Li}_2(e^{i \theta } z)- \text{Li}_2(-e^{i \theta } z)+ (1+e^{-i \theta })\log \left(\frac{1+z}{1-z}\right)\bigg ]
\end{split}
\end{equation*}
and
\begin{equation}
\begin{split}
g(z)*g_{1}(z)&=\frac{1+a}{2}g_{1}(z)-\frac{1-a}{4}\int_0^z \frac{g_{1}(t)-g_{1}(-t)}{t}\,dt\nonumber\\
&=\frac{1+a}{2} \left [\frac{e^{i \theta }}{\left(1+e^{i \theta }\right)}\,\frac{z}{1-z}-\frac{e^{i \theta } }{\left(1+e^{i \theta }\right)^2}\log \left(\frac{1+e^{i \theta } z}{1-z}\right)\right ]+\frac{(1-a)e^{i \theta }}{4\left(1+e^{i \theta }\right)^2} \bigg [\text{Li}_2(z)~ \nonumber\\
&\qquad-\text{Li}_2(-z)+\text{Li}_2(e^{i \theta } z)-\text{Li}_2(-e^{i \theta } z)-(1+e^{i \theta }) \log \left(\frac{1+z}{1-z}\right)\bigg ],
\end{split}
\end{equation}
where $\text{Li}_2(z)$ denotes the dilogarithm function defined by $\text{Li}_2(z)=\sum _{k=1}^{\infty } \frac{z^k}{k^2}$ for $|z|\leq1$.

Now the dilatation of $f*f_{1}$ is given by
\begin{equation*}
\begin{split}
\widetilde{\omega}_{1}(z)=\frac{(g*g_{1})'(z)}{(h*h_{1})'(z)}=\frac{\frac{1+a}{2}zg'_{1}(z)
-\frac{1-a}{4}\left[g_{1}(z)-g_{1}(-z)\right]}
{\frac{1+a}{2}zh'_{1}(z)+\frac{1-a}{4}\left[h_{1}(z)-h_{1}(-z)\right]}.
\end{split}
\end{equation*}

By Theorem~\ref{ThmA} with $\gamma _1=0=\gamma _2$, we obtain that $f_{1}*f$ is convex in the horizontal direction provided $f_{1}*f$ is locally univalent in $\ID$. Thus, it suffices to show that $|\widetilde{\omega}_{1}(z)|<1$ for $z\in \ID$. In view of the last expression,
$\left|\widetilde{\omega}_{1}(z)\right|<1$ is equivalent to
\begin{equation*}
\left|-\frac{1-a}{2(1+a)}\,\frac{g_{1}(z)-g_{1}(-z)}{z^2h'_{1}(z)}+e^{i\theta}\right|^2|z|^2<
\left|\frac{1-a}{2(1+a)}\,\frac{h_{1}(z)-h_{1}(-z)}{zh'_{1}(z)}+1\right|^2
\end{equation*}
and thus, it suffices to show
\begin{equation*}
\left|-\frac{1-a}{2(1+a)}\,\frac{g_{1}(z)-g_{1}(-z)}{z^2h'_{1}(z)}+e^{i\theta}\right|^2<
\left|\frac{1-a}{2(1+a)}\,\frac{h_{1}(z)-h_{1}(-z)}{zh'_{1}(z)}+1\right|^2,
\end{equation*}
or equivalently,
\begin{equation}\label{eqf3}
B(z):=\left|\frac{1-a}{2(1+a)}\,\frac{g_{1}(z)-g_{1}(-z)}{z^2h'_{1}(z)}\right|^2
-\left|\frac{1-a}{2(1+a)}\,\frac{h_{1}(z)-h_{1}(-z)}{zh'_{1}(z)}\right|^2 <\RE\left(\frac{1-a}{1+a} J(z) \right)
\end{equation}
where
\begin{equation*}
J(z)=\frac{h_{1}(z)-h_{1}(-z)}{zh'_{1}(z)}+\frac{e^{-i\theta}(g_{1}(z)-g_{1}(-z))}{z^2h'_{1}(z)}.
\end{equation*}
Because $h_{1}(0)=0=h_{1}'(0)-1$, $g_{1}(0)=0=g_{1}'(0)$ and $h_{1}'(z)\neq 0$ on $\ID$, the function $J(z)$ is clearly analytic on $\mathbb{D}$.
Since $f_1$ is convex, it follows from Lemma~\ref{lemC} that
\begin{equation*}
\left|\frac{g_{1}(z)-g_{1}(-z)}{h_{1}(z)-h_{1}(-z)}\right|<1.
\end{equation*}
Also, since
\begin{equation*}
\lim_{z\to 0}\frac{g_{1}(z)-g_{1}(-z)}{h_{1}(z)-h_{1}(-z)}=\lim_{z\to 0}\frac{g'_{1}(z)+g'_{1}(-z)}{h'_{1}(z)+h'_{1}(-z)}=0,
\end{equation*}
by Schwarz's Lemma, we  conclude that
$$\left|\frac{g_{1}(z)-g_{1}(-z)}{h_{1}(z)-h_{1}(-z)}\right|<|z|
$$
which clearly implies $B(z)<0$ for $z\in \ID$, where $B(z)$ is given by \eqref{eqf3}.
Thus,  by \eqref{eqf3}, the proof is complete if we show that $\RE\,\{J(z)\}>0$ in $\ID$. In order to do this, by \eqref{eqf1}, we need to simplify
the expression for $J(z)$ as
$$J(z) =\frac{(1+e^{i\theta}z)(1-z)}{(1+e^{i\theta})z}\left [2-\frac{(1-z)(1-e^{i\theta}z)}{(1+e^{i\theta})z}
\left \{\log\left(\frac{1+e^{i\theta}z}{1-z}\right)-\log\left(\frac{1-e^{i\theta}z}{1+z}\right)\right\}\right].
$$
By Lemma~\ref{lemD}, we need to check that $\lim_{z\to e^{it}}J(z)$ exists for all $t\in \mathbb{R}$ (where possibly $J(e^{it})=\infty$). This is clearly the case for all $t\in \mathbb{R}\backslash \{0,\pi-\theta,\pi,2\pi-\theta\}$. For the remaining values of $t$, we obtain the following limits:
$$\lim_{z\to1}J(z)=0, \quad \lim_{z\to -1}J(z)=\left\{ \begin{array}{rl}
         0 & \quad\rm{if}~ \theta=0  \\
     \infty& \quad\rm{if}~ \theta\neq 0
                          \end{array} ,
                          \right.
\quad \lim_{z\to e^{i(\pi-\theta)}}J(z)=0,
$$
and
$$\lim_{z\to e^{i(2\pi-\theta)}}J(z)=4i\tan\frac{\theta}{2}.
$$
Consequently, it suffices to show that $\RE\,\{J(e^{it})\}\geq 0$ for all $t\in \mathbb{R}\backslash \{0,\pi-\theta,\pi,2\pi-\theta\}$. We have
\begin{equation*}
\begin{split}
J(e^{it})&=\frac{(1+e^{i(\theta+t)})(1-e^{it})}{(1+e^{i\theta})e^{it}}\left[ 2-\frac{(1-e^{it})(1-e^{i(\theta+t)})}{(1+e^{i\theta})e^{it}}
\left\{\log\left(\frac{1+e^{i(\theta+t)}}{1-e^{it}}\right)\right . \right . \\
&\hspace{2cm} \left .\left .-\log\left(\frac{1-e^{i(\theta+t)}}{1+e^{it}}\right)\right\}\right]\\
&=\frac{-2i\sin\frac{t}{2}\cos\frac{\theta+t}{2}}{\cos\frac{\theta}{2}}
\left [2+\frac{2\sin\frac{t}{2}\sin\frac{\theta+t}{2}}{\cos\frac{\theta}{2}}\left \{\log\left(\frac{1+e^{i(\theta+t)}}{1-e^{it}}\right)
-\log\left(\frac{1-e^{i(\theta+t)}}{1+e^{it}}\right)\right\}\right ]
\end{split}
\end{equation*}
and hence,
\begin{equation*}
\begin{split}
\RE\,\{J(e^{it})\}
&=\frac{2\sin^{2}\frac{t}{2}\sin(\theta+t)}{\cos^{2}\frac{\theta}{2}}
\left [\arg\left(\frac{1+e^{i(\theta+t)}}{1-e^{it}}\right)
-\arg\left(\frac{1-e^{i(\theta+t)}}{1+e^{it}}\right)\right ].
\end{split}
\end{equation*}
We now introduce
$$A=\arg\left(\frac{1+e^{i(\theta+t)}}{1-e^{it}}\right) ~\mbox{ and }~B=\arg\left(\frac{1-e^{i(\theta+t)}}{1+e^{it}}\right)
$$
so that for $0\leq\theta<\pi$, we have
$$ A=\left\{ \begin{array}{rl}
        \ds \frac{\theta+\pi}{2}  & \quad\rm{if}~ t\in(0,\pi-\theta)  \\
        \ds \frac{\theta-\pi}{2}   & \quad\rm{if}~ t\in(\pi-\theta,2\pi)
                          \end{array}, \right.
~\mbox{ and }~
B= \left\{ \begin{array}{rl}
        \ds \frac{\theta-\pi}{2}  & \quad\rm{if}~ t\in(0,\pi)\cup(2\pi-\theta,2\pi)  \\
        \ds \frac{\theta+\pi}{2}  & \quad\rm{if}~ t\in(\pi,2\pi-\theta).
                          \end{array} \right.
$$
Then
\begin{equation*}
\begin{split}
A-B&=
\left\{ \begin{aligned}
        \pi  & \quad\rm{if}~ t\in(0,\pi-\theta)  \\
      -\pi  & \quad\rm{if}~ t\in(\pi,2\pi-\theta)\\
         0  & \quad\rm{if}~ t\in(\pi-\theta,\pi)\cup(2\pi-\theta,2\pi).
                          \end{aligned} \right.
\end{split}
\end{equation*}

Next, we consider $\sin(\theta+t)$. It is non-negative when $t\in(0,\pi-\theta)\cup(2\pi-\theta,2\pi)$, and negative when $t\in(\pi-\theta,2\pi-\theta)$. This observation shows that $\RE\,\{J(e^{it})\}\geq 0$ for all cases.

The general result for $\theta\in(-\pi,\pi)$ is similar to the above discussion for the case $\theta\in[0,\pi)$. This completes the proof of the assertion $|\widetilde{\omega}_{1}(z)|<1$ in $\ID$.
\end{proof}

The images of the unit disk $\mathbb{D}$ under $f*f_{1}$ for $\theta=\pi/6$ and $a=-0.5,0,0.5,0.8$ are shown in Figure~\ref{fwa}(a)-(d).
The images of $\mathbb{D}$ under $f*f_{1}$ for $a=0.5$ and $\theta=0, \pi/6, \pi/3, \pi/2$ are shown in Figure~\ref{fwn}(a)-(d).
We plot these images as equally spaced concentric radial segments.


\begin{figure}[!h]
\centering
\subfigure[$a=-0.5$]
{\begin{minipage}[b]{0.45\textwidth}
\includegraphics[height=2.4in,width=2.4in,keepaspectratio]{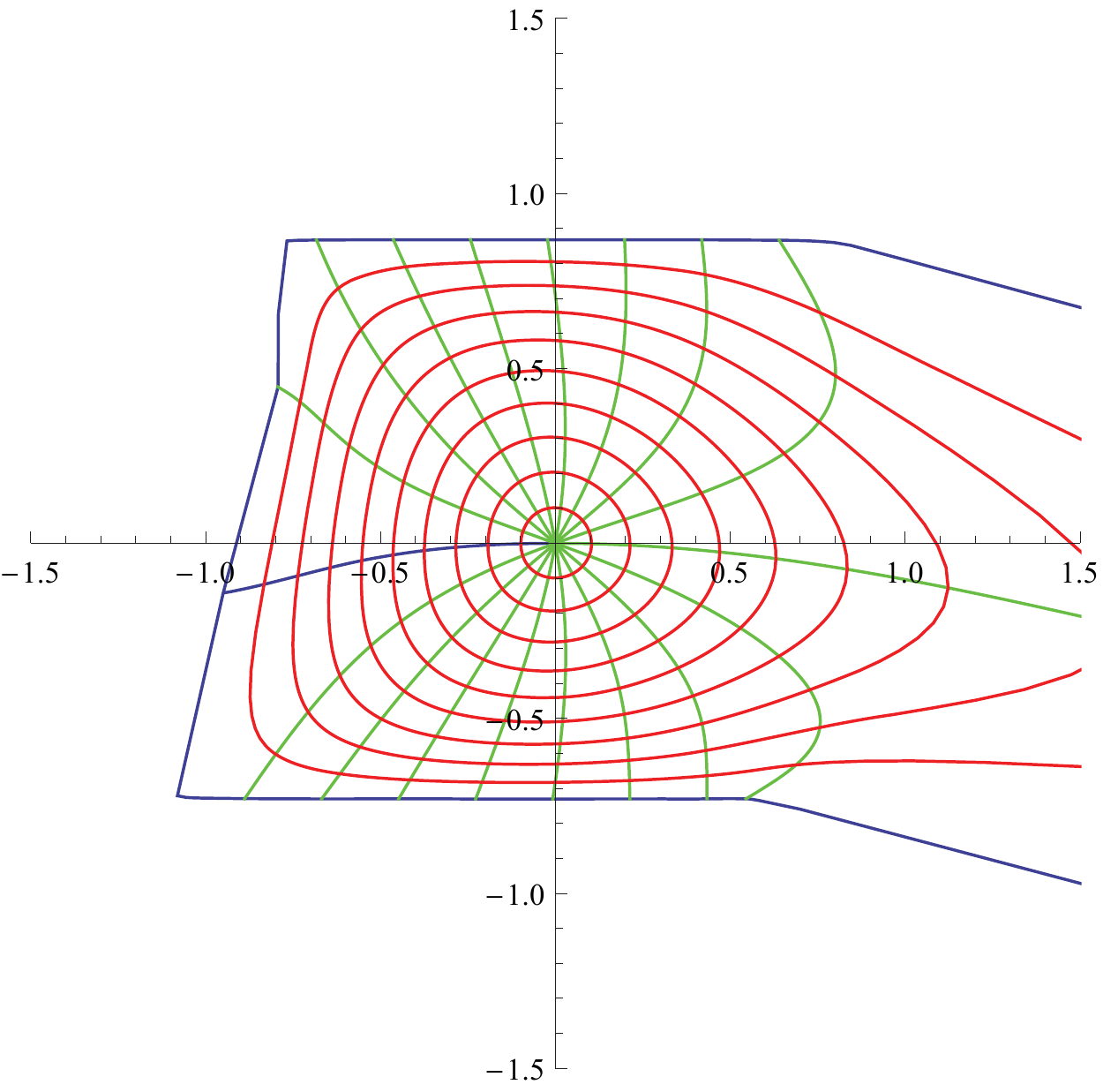}
\end{minipage}}
\subfigure[$a=0$]
{\begin{minipage}[b]{0.45\textwidth}
\includegraphics[height=2.4in,width=2.4in,keepaspectratio]{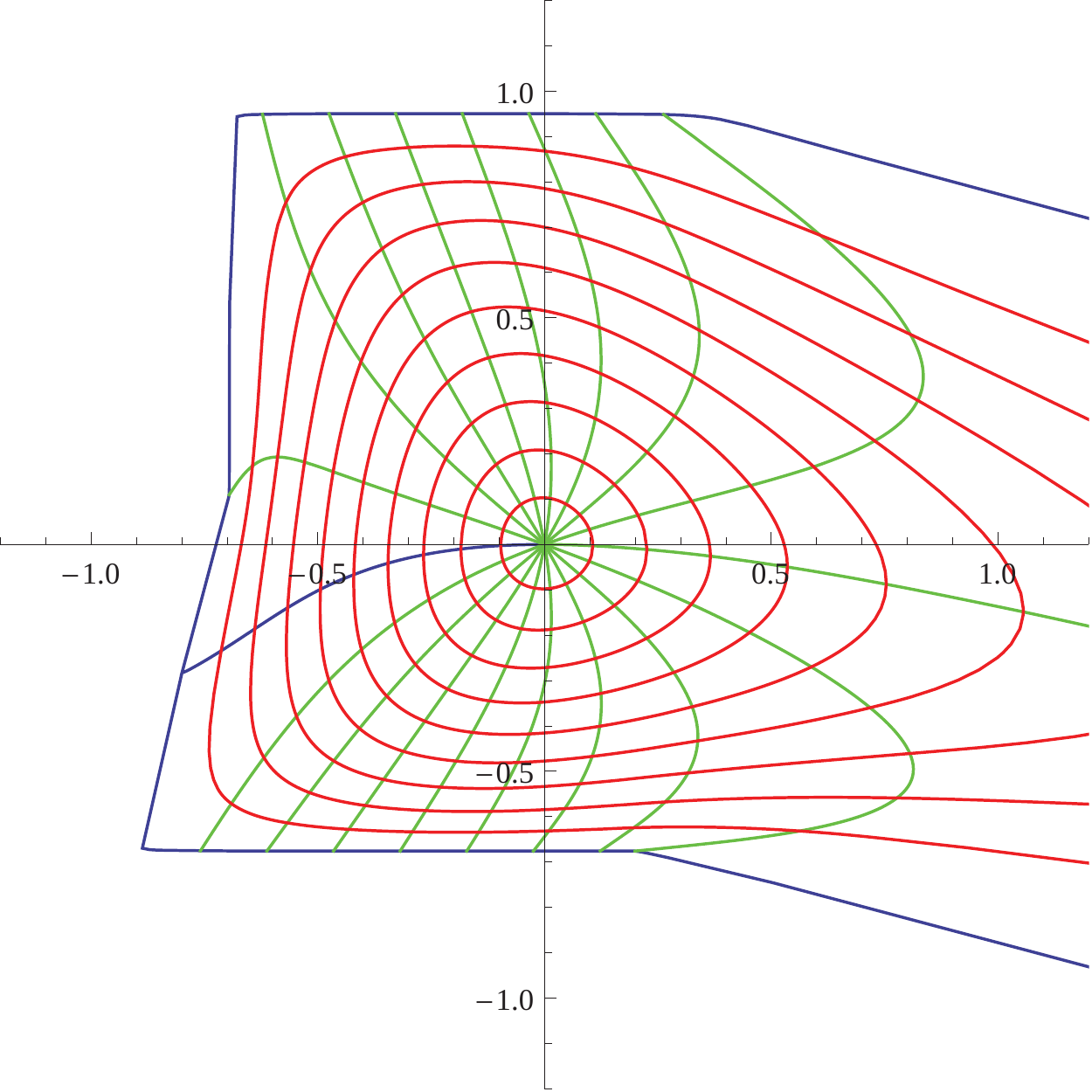}
\end{minipage}}
\subfigure[$a=0.5$]
{\begin{minipage}[b]{0.45\textwidth}
\includegraphics[height=2.4in,width=2.4in,keepaspectratio]{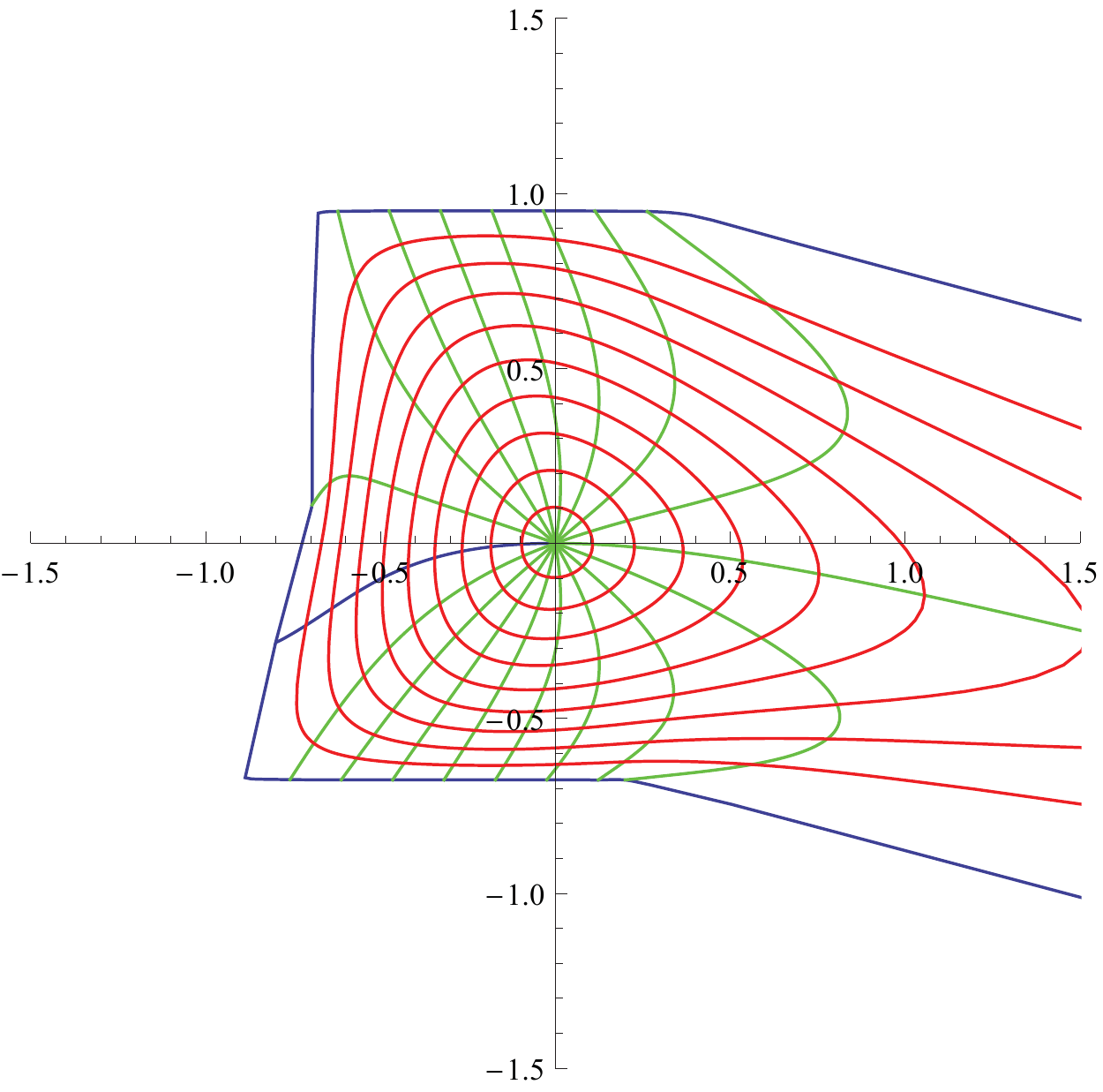}
\end{minipage}}
\subfigure[$a=0.8$]
{\begin{minipage}[b]{0.45\textwidth}
\includegraphics[height=2.4in,width=2.4in,keepaspectratio]{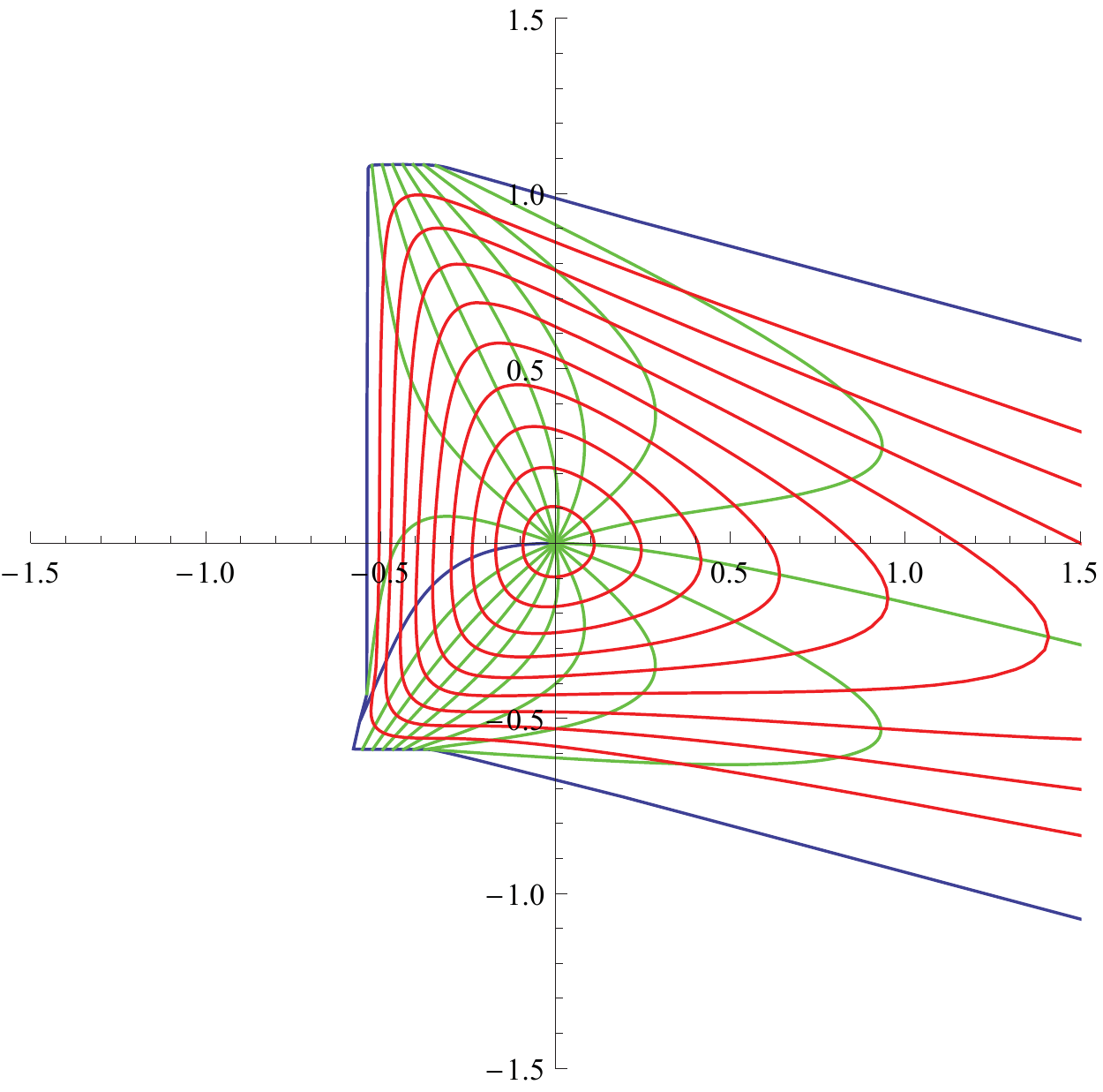}
\end{minipage}}
\caption{Images of $\mathbb{D}$ under $f*f_{1}$ for $\theta=\pi/6$ and for certain values of $a$. \label{fwa}}
\end{figure}

\begin{figure}[!h]
\centering
\subfigure[$\theta=0$]
{\begin{minipage}[b]{0.45\textwidth}
\includegraphics[height=2.4in,width=2.4in,keepaspectratio]{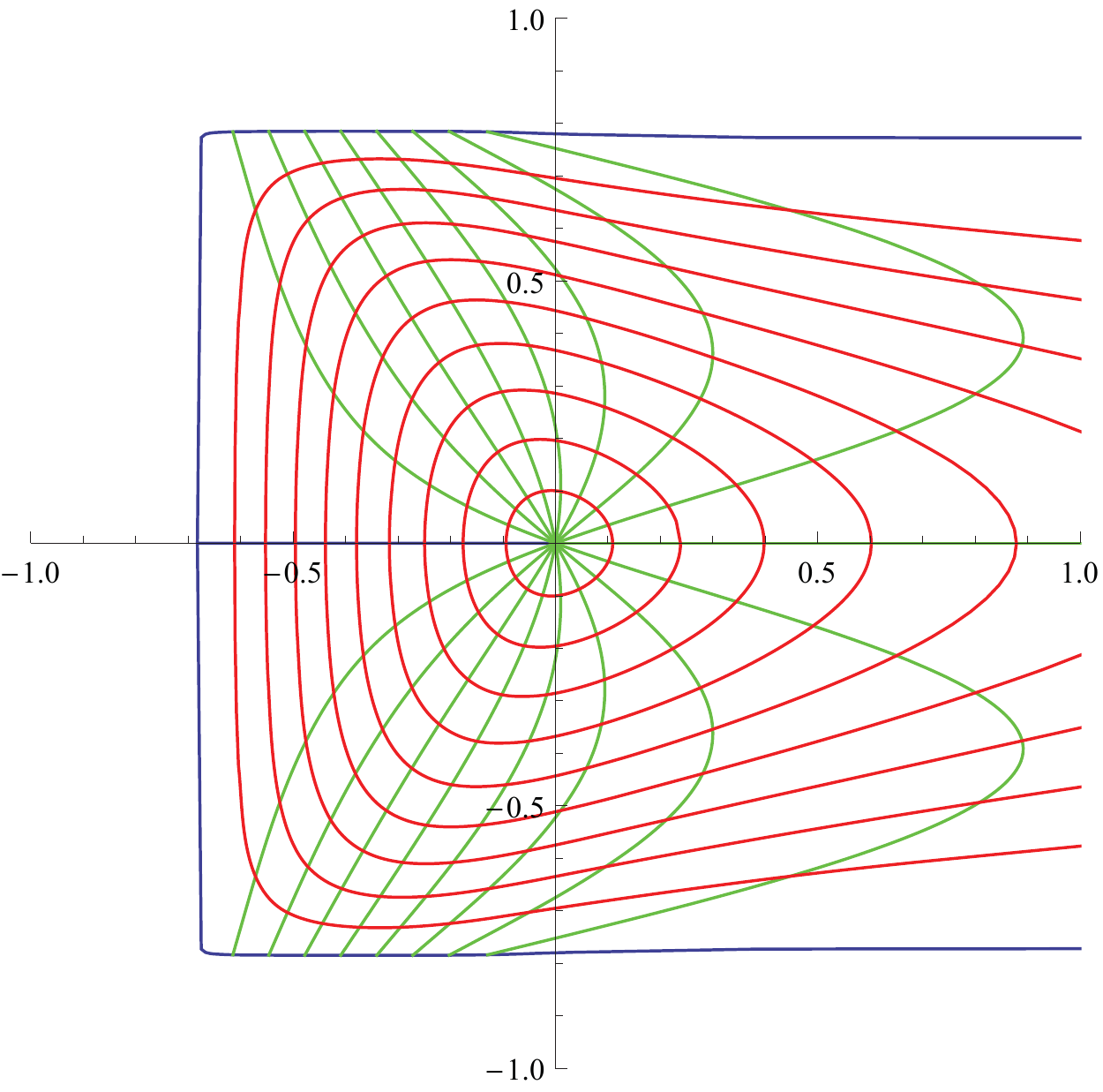}
\end{minipage}}
\subfigure[$\theta=\pi/6$]
{\begin{minipage}[b]{0.45\textwidth}
\includegraphics[height=2.4in,width=2.4in,keepaspectratio]{fa05.pdf}
\end{minipage}}
\subfigure[$\theta=\pi/3$]
{\begin{minipage}[b]{0.45\textwidth}
\includegraphics[height=2.4in,width=2.4in,keepaspectratio]{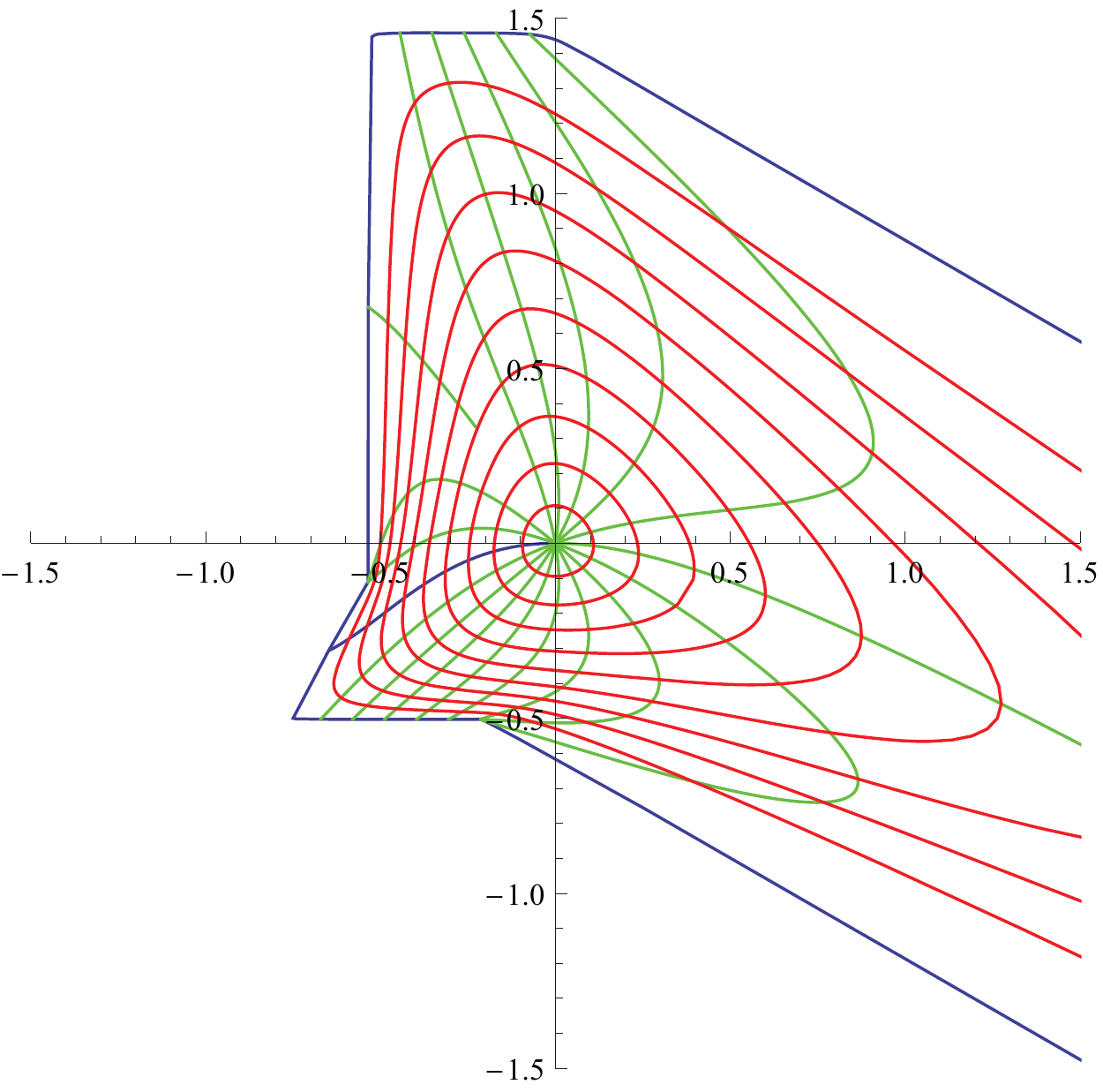}
\end{minipage}}
\subfigure[$\theta=\pi/2$]
{\begin{minipage}[b]{0.45\textwidth}
\includegraphics[height=2.4in,width=2.4in,keepaspectratio]{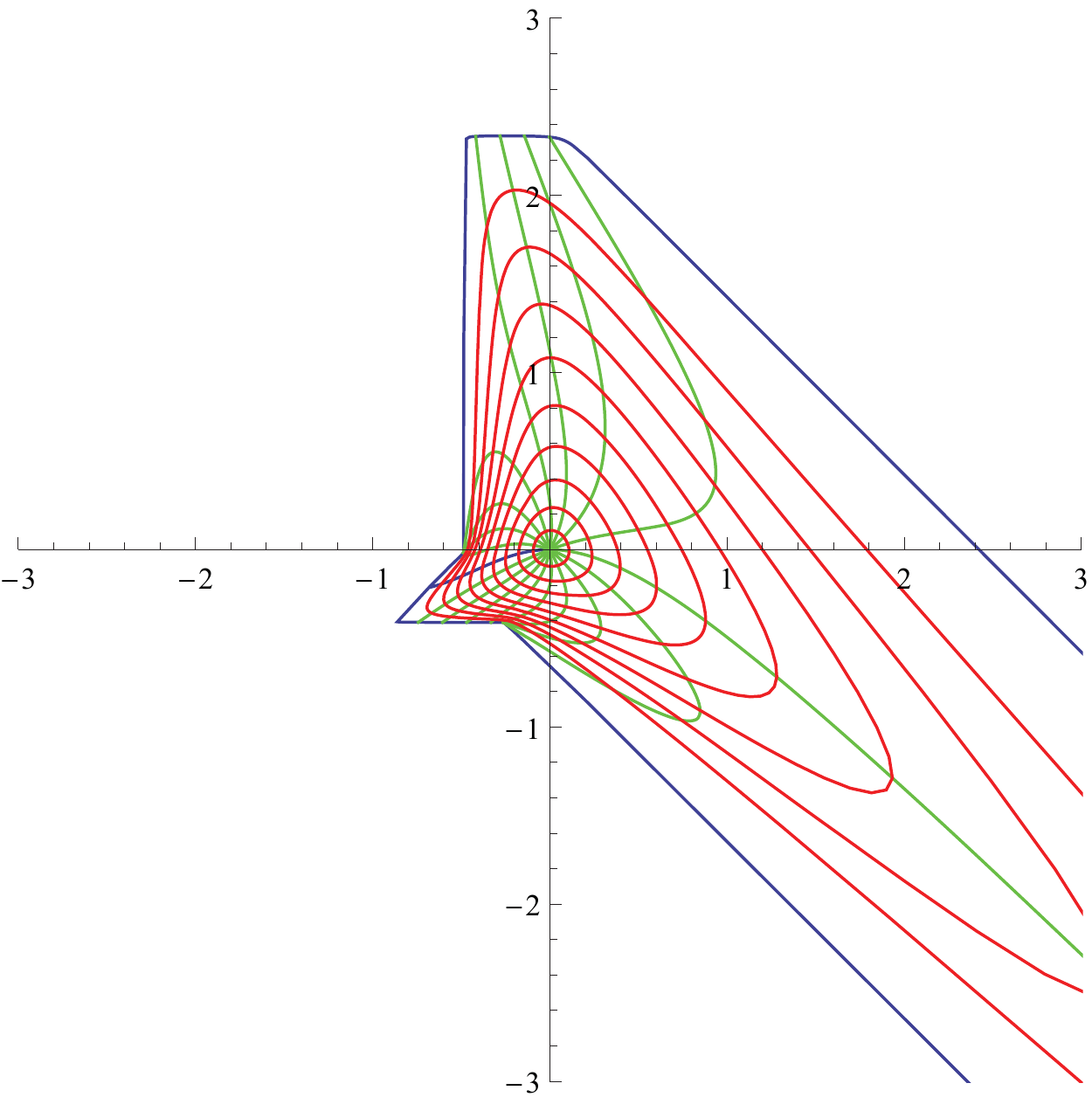}
\end{minipage}}
\caption{Images of $\mathbb{D}$ under $f*f_{1}$ for $a=0.5$ and for certain values $\theta$. }\label{fwn}
\end{figure}
\section{The Dilatations of $f*f_{n}$}
Now we compute the formulas for the dilatations of $f*f_{n}$, where
 $f=h+\overline{g}\in{\mathcal S}(H_{0})$ with $h+g=(1+a)z/(1-z)$ and the dilatation $\omega(z)=(z+a)/(1+a z)$,
where $-1<a<1$, and
$f_n\in{\mathcal S^0}(H_{0})$ with the  dilatation $\omega_{n}(z)=e^{i\theta}z^{n}~(\theta\in\mathbb{R},n\in\mathbb{N})$.
First, we begin by computing the representation of $f_{n}$ when $\theta=\pi$.  Note that in this case
$$h'_{n}(z)+g'_{n}(z)=\frac{1}{(1-z)^2}\quad{\rm and}\quad g'_{n}(z)=-z^nh'_{n}(z)
$$
and thus,
$$h'_{n}(z)=\frac{1}{(1-z^n)(1-z)^2}.
$$
In order to compute $h_n(z)$, we may rewrite it as
\begin{equation*}
\begin{split}
h'_{n}(z) 
&=\frac{n^2-1}{12 n}\,\frac{1}{1-z}+\frac{n-1}{2n}\,\frac{1}{(1-z)^2}+\frac{1}{n}\,\frac{1}{(1-z)^3}
+\frac{1}{n}\sum_{k=1}^{n-1}\frac{1}{(1-e^{\frac{2k\pi}{n}i})^2}\frac{1}{(1-ze^{-\frac{2k\pi}{n}i})}.
\end{split}
\end{equation*}
By integrating the previous expression we arrive at
\begin{equation}\label{eqhgn}
\begin{split}
h_{n}(z)&=\frac{n-1}{2 n}\,\frac{z}{1-z}+\frac{1}{2n}\,\frac{z (2-z)}{ (1-z)^2}-\frac{n^2-1}{12 n}\log (1-z)+\frac{1}{4 n}\sum _{k=1}^{n-1} \csc ^2\frac{\pi  k}{n} \log \left(1-z e^{-\frac{ 2k \pi  }{n}i}\right),\\
g_{n}(z)&=\frac{n+1}{2 n}\,\frac{z}{ 1-z}-\frac{1}{2n}\,\frac{z (2-z)}{ (1-z)^2}+\frac{n^2-1}{12 n}\log (1-z)-\frac{1}{4 n}\sum _{k=1}^{n-1} \csc ^2\frac{\pi  k}{n} \log \left(1-z e^{-\frac{ 2k \pi  }{n}i}\right).
\end{split}
\end{equation}
From \eqref{eqhg} and \eqref{eqhgC}, we obtain  $f*f_{n}$ from
 \begin{equation}\label{eqCv}
\begin{split}
h(z)*h_{n}(z)&=\frac{1-a}{4}\int\frac{h_{n}(z)-h_{n}(-z)}{z}dz+\frac{1+a}{2}h_{n}(z),\\
g(z)*g_{n}(z)&=-\frac{1-a}{4}\int\frac{g_{n}(z)-g_{n}(-z)}{z}dz+\frac{1+a}{2}g_{n}(z),
\end{split}
\end{equation}
and  the dilatation $\widetilde{\omega}_{n}$ of $f*f_{n}$ is then given by
\begin{equation}\label{eqwn}
\begin{split}
\widetilde{\omega}_{n}(z)=\frac{(g*g_{n})'(z)}{(h*h_{n})'(z)}=\frac{-\frac{1-a}{4}\left(g_{n}(z)-g_{n}(-z)\right)+\frac{1+a}{2}zg'_{n}(z)}
{\frac{1-a}{4}\left(h_{n}(z)-h_{n}(-z)\right)+\frac{1+a}{2}zh'_{n}(z)}.
\end{split}
\end{equation}
Table~1 presents values of $|\widetilde{\omega}_{n}(z)|$ for certain choices of $n,a$ and $z$ with the help of \emph{Mathematica}. 
From Table~1, we observe for some $n\geq 2$ that $\left|\widetilde{\omega}_{n}(z)\right|>1$ for certain values of $a$ and $z$.

\begin{table}[!h]
\begin{tabular}{|c|c|c|c|}
\hline
\hline
$n$ & $a$ & $z$ &\textbf{$\left|\widetilde{\omega}_{n}(z)\right| $}\\
\hline
2   &0.5  & $0.99 \exp \left(\frac{ \pi }{ 3} i\right)$  &1.06019  \\
\hline
3   &0.5  & $0.99 \exp \left(\frac{3 \pi }{4} i\right)$ &1.28884  \\
\hline
4   &-0.5 & $0.99 \exp \left(\frac{\pi } {8}i  \right)$ &1.07326  \\
\hline
5   &-0.5 & $0.99 \exp \left(\frac{\pi }{10}i \right)$ &1.04422  \\
\hline
6   &-0.4 & $0.99 \exp \left(\frac{\pi }{11}i \right)$ &1.03038  \\
\hline
7   &0.5  & $0.99 \exp \left(\frac{\pi }{3}i  \right)$  &1.04396  \\
\hline
8   &0.5  & $0.99 \exp \left(\frac{\pi }{3}i  \right)$  &1.02052  \\
\hline
9   &0.5   & $0.99 \exp \left(\frac{\pi }{2}i  \right)$  &1.12641  \\
\hline
10 &0.3   & $0.99 \exp \left(\frac{\pi }{4}i \right)$  &1.05563  \\
\hline
11 &-0.7  & $0.99 \exp \left(\frac{\pi }{5}i \right)$  &1.32055  \\
\hline
12 &0     & $0.99 \exp \left(\frac{\pi }{5}i \right)$  &1.09197  \\
\hline
13 &0     & $0.99 \exp \left(\frac{\pi }{5}i \right)$  &1.00698  \\
\hline
14 &-0.4 & $0.99 \exp \left(\frac{\pi }{6}i \right)$  &1.20222  \\
\hline
15 &-0.2 & $0.99 \exp \left(\frac{\pi }{6}i \right)$  &1.04876  \\
\hline
\end{tabular}
\vskip.10in
\caption{Values of $\left|\widetilde{\omega}_{n}(z)\right|$ for certain values $a$ and $z$.}
\end{table}

Next, we assume that $\theta\neq\pi$ and consider
$$h_{n}'(z)+g_{n}'(z)=\frac{1}{(1-z)^2}\quad{\rm and}\quad\omega_{n}(z)=\frac{ g'_{n}(z)}{h'_{n}(z)}=e^{i\theta}z^{n},
$$
so that
\begin{equation*}
\begin{split}
h_{n}'(z)&=\frac{1}{(1+e^{i\theta}z^{n})(1-z)^2}\\
&=\frac{ne^{i\theta}}{(1+e^{i\theta})^2}\,\frac{1}{1-z}+\frac{1}{1+e^{i\theta}}\,\frac{1}{(1-z)^2}\\
&\qquad-\frac{1}{n}\sum_{k=0}^{n-1}
\frac{1}{\left(1-e^{i\frac{(2k+1)\pi-\theta}{n}}\right)^2}\frac{1}{1-z e^{-i\frac{(2k+1)\pi-\theta}{n}}}.
\end{split}
\end{equation*}
Integration gives
\begin{equation}\label{eqhn1}
\begin{split}
h_{n}(z)&=-\frac{ne^{i\theta}}{(1+e^{i\theta})^2}\log(1-z)+\frac{1}{1+e^{i\theta}}\, \frac{z}{1-z}\\
&\quad+\frac{1}{4n}\sum_{k=0}^{n-1}\csc^2\frac{(2k+1)\pi-\theta}{2n}\log\left(1-z e^{-i\frac{(2k+1)\pi-\theta}{n}}\right),
\end{split}
\end{equation}
and thus
\begin{equation}\label{eqgn1}
\begin{split}
g_{n}(z)&=\frac{z}{1-z}-h_{n}(z)\\
&=\frac{ne^{i\theta}}{(1+e^{i\theta})^2}\log(1-z)+\frac{e^{i\theta}}{1+e^{i\theta}}\, \frac{z}{1-z}\\
&\quad-\frac{1}{4n}\sum_{k=0}^{n-1}\csc^2\frac{(2k+1)\pi-\theta}{2n}\log\left(1-z e^{-i\frac{(2k+1)\pi-\theta}{n}}\right).
\end{split}
\end{equation}
Substituting the expressions \eqref{eqhn1} and \eqref{eqgn1} into \eqref{eqwn},  we get Table~2.
\begin{table}[!h]
\begin{tabular}{|c|c|c|c|c|}
\hline
$n$ & $a$ & $\theta$ & $z$ &$\left|\widetilde{\omega}_{n}(z)\right| $\\
\hline
2   &0.5  & $\frac{ \pi }{8}$ & $0.99 \exp \left(\frac{ \pi }{2} i\right)$  &1.16334  \\
\hline
3   &0.5  & $\frac{ \pi }{12}$ & $0.99 \exp \left(\frac{ \pi }{2} i\right)$ &1.09124  \\
\hline
4  &0.5  & $\frac{ \pi }{3 }$ & $0.99 \exp \left(\frac{ \pi }{3}i \right)$  &1.05616  \\
\hline
5  &0.8  & $\frac{ \pi }{6 }$ & $0.99 \exp \left(\frac{2\pi }{3}i \right)$  &1.06377  \\
\hline
6  &0.7  & $\frac{ \pi }{3 }$ & $0.99 \exp \left(\frac{ \pi }{2}i \right)$  &1.09271  \\
\hline
7  &0.7  & $\frac{ \pi }{6 }$ & $0.99 \exp \left(\frac{ \pi }{2}i \right)$  &1.01364  \\
\hline
8  &0.6  & $-\frac{\pi }{3 }$ & $0.99 \exp \left(\frac{ \pi }{2}i   \right)$  &1.04091  \\
\hline
9  &0.7  & $\frac{\pi }{2 }$ & $0.99 \exp \left(-\frac{ 7\pi }{8}i \right)$  &1.20496  \\
\hline
10&0.7  & $-\frac{\pi }{2 }$ & $0.99 \exp \left(-\frac{ 7\pi }{8}i \right)$  &1.97405  \\
\hline
11&0.4  & $\frac{\pi }{2 }$ & $0.99 \exp \left(-\frac{ 7\pi }{8}i \right)$  &1.42585  \\
\hline
12&0   & $\frac{\pi }{2 }$ & $0.99 \exp \left(\frac{ 7\pi }{8}i \right)$  &1.09957  \\
\hline
13&0.9 & $-\frac{\pi }{16}$ & $0.99 \exp \left(\frac{ 7\pi }{8}i \right)$  &1.01078  \\
\hline
14&0.9 & $-\frac{3\pi }{4}$ & $0.99 \exp \left(-\frac{ 7\pi }{8}i \right)$  &1.08478  \\
\hline
15&0.9 & $-\frac{\pi }{4}$ & $0.99 \exp \left(\frac{ 7\pi }{8}i \right)$  &1.00032  \\
\hline
\end{tabular}
\vskip.10in
\caption{Values of $\left|\widetilde{\omega}_{n}(z)\right|$ for certain values $a,\theta$ and $z$.}
\end{table}

From the above discussion and some complicated calculations and experiments, we propose the following.
\begin{problem}
Let $f=h+\overline{g}\in{\mathcal S}(H_{0})$ with $h+g=(1+a)z/(1-z)$ and the dilatation $\omega(z)=(z+a)/(1+a z)$, where $-1<a<1$,
and $f_n =h_{n}+\overline{g_{n}}\in{\mathcal S^0}(H_{0})$ with dilatations $\omega_{n}(z)=e^{i\theta}z^n~(\theta\in\mathbb{R},n\in\mathbb{N})$.
We conjecture that $f*f_{n}$ is not locally univalent for $n\geq 2$. In this situation, it is natural to ask for
the radius of univalency of $f*f_{n}$.
\end{problem}



\subsection*{Acknowledgements}
The authors thank the referee for pointing out the error occurred in Problem 1 in page 2, and the comments which essentially improved the article.
The research of the first author was supported by the Project Education Fund of Yunnan Province under Grant No. 2015Y456, the
First Bath of Young and Middle-aged Academic Training Object Backbone of Honghe University under Grant No. 2014GG0102. The work of the first author
was completed during his visit to Indian Statistical Institute, Chennai Centre.





\end{document}